\documentclass[12pt, a4paper]{article}

\usepackage[T1]{fontenc}
\usepackage[utf8]{inputenc}
\usepackage{amssymb}
\usepackage{amsmath}
\usepackage{amsfonts}
\usepackage{amsthm}
\usepackage{latexsym}
\usepackage[mathscr]{eucal}
\usepackage{multirow}
\usepackage{color}
\usepackage{graphicx}
\usepackage{bbm}
\usepackage{soul}
\usepackage[normalem]{ulem}

\newtheorem{thm}{Theorem}[section] 
\newtheorem{theorem}[thm]{Theorem}

\newtheorem{lem}[thm]{Lemma} 
\newtheorem{lemma}[thm]{Lemma} 
\newtheorem{prop}[thm]{Proposition}
\newtheorem{cor}[thm]{Corollary}
\newtheorem{definition}{Definition}

 \theoremstyle{remark}
\newtheorem*{remark*}{Remark} \newtheorem{remark}{Remark}

 \theoremstyle{definition}
 
\renewcommand{\leq}{\leqslant} \renewcommand{\le}{\leq}
\renewcommand{\geq}{\geqslant} \renewcommand{\ge}{\geq}

\numberwithin{equation}{section}

\DeclareMathOperator{\dist}{dist}
\DeclareMathOperator{\diam}{diam}

\def\({\left(} 
\def\){\right)} 
\def\r){\right)} 
\def\[{\left[}
\def\]{\right]} 
\def\<{\langle} 
\def\>{\rangle}

\def \EE {\mathbb{E}} 
\def \PP {\mathbb{P}} 
\def \RR {\mathbb{R}}

\def \pK {\mathcal{K}}

\def \CI {C_1} 
\def \CII {C_3} 
\def \CIII {C_3} 

\def \CIV {C_4} 
\def \CV {C_5} 
\def \CVI {C_6} 
\def \CVII {C_7}
\def \CVIII {C_8} 
\def \CIX {C_9} 
\def \CX {C_{10}} 
\def \CXI {C_{11}}
\def \CXII {C_{12}} 
\def \CXIII {C_{13}} 
\def \CXIV {C_{14}} 
\def \CXV {C_{2}}
\def \CXVI {C_{16}}
\def \CXVII {C_{15}}
\def \CXVIII {C_{17}}
\def \CXIX {C_{18}}
\def \CXX {C_{1}}
\def \kk {G}
\def \tp{\tilde{p}}
\def \tEE {\tilde{\mathbb{E}}} 
\def \tPP {\tilde{\mathbb{P}}}
\def \PP {\mathbb{P}}

\def \tG {\tilde G}
\def \tL {\tilde{\mathcal{L}}}

\def \tGDD {{\tilde G_D}}
\def \tP {\tilde P}

\def\deltaDD {\delta_D}

 \newcommand{\R}{\mathbb{R}}

\newcommand{\E}{\mathbb{E}}

\newcommand{\WUSC}[3]{\textrm{\rm WUSC}(#1,#2,#3)}
\newcommand{\WLSC}[3]{\textrm{\rm WLSC}(#1,#2,#3)}
\newcommand{\lC}{{\underline{c}}}
\newcommand{\uC}{{\overline{C}}}
\newcommand{\la}{{\underline{\alpha}}}
\newcommand{\ua}{{\overline{\alpha}}}
\newcommand{\lt}{{\underline{\theta}}}
\newcommand{\ut}{{\overline{\theta}}}

\newcommand{\CG}{\mathcal{G}}

\newcommand{\de}{\delta}
\newcommand{\dex}{\delta_x}
\newcommand{\dey}{\delta_y}
\newcommand{\dez}{\delta_z}

\title{
	Green function for gradient perturbation of  unimodal L\'evy processes in  the real line
}

\author{
	T. Grzywny\thanks{The first author was supported by the statutory fund of the Department of Mathematics,
Faculty of Pure and Applied Mathematics, Wroc\l{}aw University of Science
and Technology.}, 
	T. Jakubowski \thanks{The second and third authors were partially supported by NCN grant 2015/18/E/ST1/00239}
		and 	
	G. \.{Z}urek\\
	Faculty of Pure and Applied Mathematics\\Wroc\l{}aw University of Science and Technology, Poland
 }

\date{}
\begin{document}

\maketitle

\begin{abstract}
We prove that the Green function of a generator of symmetric unimodal L\'{e}vy processes with the weak lower scaling order
bigger than one and the Green function of its gradient perturbations are comparable
for bounded $C^{1,1}$ subsets of the real line if the drift function is from  an appropriate Kato class.
\end{abstract}

\section{Introduction}

Perturbations of Markovian generators are widely studied from many years. This theory may be considered from various points of view. Such perturbations appear, e.g., in local and non-local partial differential equations \cite{MR3238505, MR2680400, MR3060702, MR3017289}, semigroup theory \cite{MR2283957, MR3035054, MR2875353, MR3514392, MR2457489}, stochastic processes \cite{MR2369047, MR1310558, MR1310558}, potential theory \cite{MR2892584, MR3050510, MR3129851}. One of the natural question is: how this perturbation affects the solutions of the equations related to the unperturbed operator (e.g., the transition density of the semigroup, the Green function).

In this paper we are interested in the gradient perturbations and the potential theory of the perturbed operator. We briefly recall some results closely related to our research.  	Cranston and Zhao in \cite{1987-MC-ZZ-cmp} considered  the operator $\Delta +b(x) \nabla$ in $\RR^d$ for $d \ge 2$. They proved that the Green function and the harmonic measure of Lipschitz domains are comparable with those of $\Delta$ for the drift $b$ from the appropriate Kato class. In \cite{MR2353039} and \cite{MR2369047} Jakubowski studied the $\alpha$-stable Ornstein-Uhlenbeck process. He proved estimates for the first exit time from the ball and Harnack inequality for this process. In \cite{MR2892584} Bogdan and Jakubowski proved similar results as Cranson and Zhao for $\Delta^{\alpha/2} + b(x) \nabla$ in $C^{1,1}$ domains in $\RR^d$, $d \ge 2$. In the recent paper \cite{2017-TG-TJ-GZ-pms} these results were generalized to the case of pure-jump symmetric unimodal L\'evy processes possessing certain weak scaling properties. We note that in the papers \cite{MR2892584,2017-TG-TJ-GZ-pms} the case $d=1$ was omitted. The aim of this paper is to fill this gap and prove analogous results in one dimensional case. 

We will denote by $\{X_t\}$ a pure-jump symmetric  unimodal L\'{e}vy process on $\R$. That is, a process with the symmetric density function $p_t(x)$ on $\R \setminus \{0\}$ which is non-increasing on $\R_+$. The characteristic exponent of $\{X_t\}$ equals
$$
	\psi(x) = \int_{\R} \(1 -\cos ( x  z )\right) \nu(dz), \,\quad x\in\R.
$$
where $\nu$ is a L\'{e}vy measure, i.e., $\int_{\R}\(1\wedge|z|^2\right)\nu(dz)<\infty$. For general information on unimodal processes, we refer the reader to \cite{MR3165234, MR3225805, MR705619}. A primary example of the mentioned class of processes is the symmetric $\alpha$-stable L\'evy process having the fractional Laplacian $\Delta^{\alpha/2}$ as a generator.

Let
\begin{equation}\label{generator}
{\cal L} f(x)= \int_{\R} \left(f(x+z)-f(x)
-{\bf 1}_{|z|<1}(z\cdot \nabla f(x))\right)\nu(dz)\,, \quad f\in C^2_b(\R)\,,
\end{equation}
be a generator of the process $X_t$. We will consider a non-empty bounded open $C^{1,1}$ set $D$ and the Green function $G_D(x,y)$ for $\cal{L}$. Now, let $\tilde G_D(x,y)$ be a Green function for
$$
\tilde{\mathcal L} = \mathcal{L} + b(x) \cdot \nabla,
$$
where $b$ is a function from the Kato class  $\pK_1$ (see Section 2 for details). Our main result is

\begin{theorem}
\label{Theorem1}
Let $b\in \pK_1$, and  let $D\subset \R$ be 
an union of finitely many open intervals with positive distance between every two intervals. We assume that the characteristic exponent
$$
\psi\in \WLSC{\la}{0}{\lC}\cap \WLSC{\la_1}{1}{\lC_1} \cap \WUSC{\ua}{0}{\uC}, \qquad  \mbox{where }\la_1>1,
$$
Then, there exists a constant $C$ such that for $x,y \in D$,
  \begin{equation}
    \label{eq:egf}
C^{-1}G_D(x,y) \le \tilde G_D(x,y) \le C G_D(x,y)\,.
  \end{equation}
\end{theorem}
Here WLSC and WUSC are the classes of functions satisfying a weak lower and weak upper scaling condition, respectively (see Section 2 for definitions). Set $D$ should be considered as an one-dimensional case of a bounded $C^{1,1}$ set, see Definition \ref{def:C11}. 

Generally, we follow the approach of \cite{MR2892584} and \cite{2017-TG-TJ-GZ-pms}, however there are some important differences. Although the geometry of the set $D$ is much simpler than in higher dimensions, it seems that the one dimensional case sometimes demands more delicate arguments. One of the main difficulties are the proper estimates for derivative of the Green function. As it was mentioned in \cite{MR2892584}, for $d=1$, the available estimates  
\begin{align}\label{eq:GradEst_weak}
	|\partial_x G_D(x,y)|\le c G_D(x,y) /(\delta_D(x) \land |x-y|) 
\end{align}
are not integrable near $y$. The estimates \eqref{eq:GradEst_weak} hold, i.e. if $\nu'(r)/r$ is non-increasing (see \cite[Lemma 3.2]{2017-TG-TJ-GZ-pms} and \cite[Theorem 1.4]{TKMR}). To overcome this difficulty, we improve the estimates \eqref{eq:GradEst_weak} near the pole in $y$, see Theorem \ref{thm:GradEst}. This result is new even for the fractional Laplacian. We emphasize here that we make no additional assumption on the monotonicity of $\nu'(r)/r$ as mentioned above. 
Like in the mentioned papers, our mail tool is the perturbation formula. First, we use it to obtain estimates for sets $D$ with a small radius. Since the Green function $G_D(x,y)$ is bounded, we do not use the perturbation series as in \cite{MR2892584} and \cite{2017-TG-TJ-GZ-pms}. Instead, we propose a simpler iteration argument.

We note also that one of our standing assumptions is $\la_1>1$. It may be understood that the rank of the operator $\mathcal{L}$ is larger than 1. Without this assumption the drift term may have the stronger effect than $\mathcal{L}$ on the behavior of the Green function of the $\tilde{\mathcal L}$. Any results concerning the cases $\la \le 1$ would be interesting, however for $\la <1$, Theorem  \ref{Theorem1} cannot hold in the form above (see the Introduction of \cite{MR2892584} for more details)

The paper is organized as follows. In Section 2 we define the process $X$ and present its basic properties. In Section 3, we introduce the Green function of $X$, prove the estimates for its derivative and some 3G-like inequalities.  In Section 4, we define the operator $\tilde{\mathcal{L}}$ and the Green function of the underlying  Markov process. Lastly, in Section 5, we prove Theorem \ref{Theorem1}.

When we write $f(x) \stackrel{C}{\approx} g(x)$, we mean that there is a number $0 < C < \infty$ independent of $x$, i.e.
a constant, such that for every $x$, $C^{-1} f (x) \le g(x) \le C f (x)$. If the value of $C$ is not important we simply write $f(x) \approx g(x)$. The notation 
$C = C(a, b, \ldots, c)$ means that $C$ is a constant which depends only on $a, b, \ldots , c$.

We use a convention that numberded constants denoted by capital letters do not change throughout the paper. 
For a symmetric function  $f:\R\rightarrow [0,\infty)$ we shall often write $f(r)=f(x)$ for any $x \in \R$ with $|x| = r$.
\section{Preliminaries}
In what follows, $\RR$ denotes the Euclidean space of real numbers, $dy$ stands for the Lebesgue measure on $\R$. Without further mention we will only consider Borelian sets, measures and functions in $\RR$. As usual, we write $a \land b = \min(a,b)$ and $a \vee b = \max(a,b)$. 
We let $B(x,r)=\{y\in \RR \colon |x-y|<r\}$.
For the arbitrary set $D\subset \RR$, the distance to the complement of $D$, will be denoted by
$$\dex =\dist(x,D^c)\,.$$

\begin{definition}
Let $\lt\in [0,\infty)$ and
 $\phi$ be a non-negative non-zero  function on $(0,\infty)$.
We say that
$\phi$ satisfies {the} {\it \bfseries weak lower scaling condition} (at infinity) if there are numbers
$\la>0$ and  $\lC \in(0,1]$  such that
\begin{equation}\label{eq:LSC}
 \phi(\lambda\theta)\ge
\lC\lambda^{\,\la} \phi(\theta)\quad \mbox{for}\quad \lambda\ge 1, \quad\theta>\lt.
\end{equation}

In short, we say that $\phi$ satisfies WLSC($\la, \lt,{\lC}$) and write $\phi\in\WLSC{\la}{ \lt}{\lC}$.
If $\phi\in\WLSC{\la}{0}{\lC}$, then we say
that $\phi$ satisfies the {\bfseries  \emph{global} weak lower scaling condition}.

Similarly, we consider
 $\ut\in [0,\infty)$.
The  {\it \bfseries weak upper scaling condition} holds if there are numbers $\ua<2$
and $\uC{\in [1,\infty)}$ such that
\begin{equation}\label{eq:USC}
 \phi(\lambda\theta)\le
\uC\lambda^{\,\ua} \phi(\theta)\quad \mbox{for}\quad \lambda\ge 1, \quad\theta> \ut.
\end{equation}
In short, $\phi\in\WUSC{\ua}{ \ut}{\uC}$. For {\bfseries  \emph{global} weak upper scaling} we require $\ut=0$ in \eqref{eq:USC}.
\end{definition}

Throughout the paper, $X_t$ will be the pure-jump symmetric unimodal L\'evy process on $\RR$. The L\'evy measure $\nu$ of $X_t$ is symmetric and non-increasing, so it admits the density $\nu$, i.e., $\nu(dx) = \nu(|x|)dx$.   Hence  the characteristic exponent $\psi$ of $X_t$ is radial as well.   

We assume that (see Theorem \ref{Theorem1})
\begin{align}
\psi  & \in \WLSC{\la}{0}{\lC} \cap \WUSC{\ua}{0}{\uC}\,, \label{eq:assum1}\\
\psi  & \in  \WLSC{\la_1}{1}{\lC_1}, \qquad  \mbox{for some } \la_1>1\,. \label{eq:assum2}
\end{align}

\noindent Following \cite{MR632968}, we define
$$
h(r)=\int_\R \(1\wedge \frac{|x|^2}{r^2}\)\nu(|x|)dx, \qquad r>0\,.
$$
Let us notice that
$$
h(\lambda r)\leq h(r)\leq \lambda^2  h(\lambda r), \quad \lambda>1.
$$
Moreover, by \cite[Lemma 1 and (6)]{MR3165234}
$$
2^{-1}\psi(1/r)\leq h(r)\leq \CXX \psi(1/r).
$$
Here, we may choose $\CXX = \pi^2/2$ but it will be more convenient to write this constant as $\CXX$.
We define the function $V$ as follows,	
$$V(0)=0 \,\,\,\mathrm{and}\,\,\, V(r)=1/\sqrt{h(r)}, \quad r>0.$$
Since $h(r)$ is non-increasing, $V$ is non-decreasing. We have
\begin{equation}\label{subadd} V(r)\leq V(\lambda r)\leq \lambda V(r),\quad r\geq 0 ,\,\lambda >1.
\end{equation}

\noindent By weak scaling properties of $\psi$ and the property $h(r) \approx \psi(1/r)$, we get
\begin{equation}\label{scalV1} \(\frac{\lC}{2\CXX}\)^{1/2}\lambda^{\la/2}\leq \frac{V(\lambda r)}{V(r)}\leq (2\uC  \CXX)^{1/2}\lambda^{\ua/2},\quad r> 0 ,\,\lambda >1,
\end{equation}
\begin{equation}\label{scalV2}
\frac{V(\eta r)}{V(r)} \leq \(\frac{2\CXX}{\lC_1}\)^{1/2}\eta^{\la_1/2},\quad \,\eta<1,\,r<1 .
\end{equation}
Therefore,
$ V   \in \WLSC{\la/2}{0}{\sqrt{\lC/(2\CXX)}} \cap \WUSC{\ua/2}{0}{\sqrt{2\uC\CXX}}$.\\

\begin{remark}\label{remScalExp}
The threshold $(0,1)$ in scaling of $V$ in \eqref{scalV2} may be replaced by any bounded interval at the expense of constant $\sqrt{2\CXX/\lC_1}$ (see  \cite[Section 3]{MR3165234}), i.e., for any $R>1$, there is a constant $c$ such that
\begin{equation}\label{scalV2R}
\frac{V(\eta r)}{V(r)} \leq \ c \eta^{\la_1/2},\quad \,\eta<1,   r<R .
\end{equation}
\end{remark}

\noindent We define  
$$M(r) = \frac{V^2(r)}{r^2},\,\,\,r > 0.$$
We note that $M(\cdot)$ is decreasing and $\lim_{r \to 0^+} M(r) = \infty$. To simplify the notations how the constants depend on the parameters, we put
$$\sigma=(\ua,\uC,\la,\lC)\quad \text{ and } \quad \underline{\sigma}=(\sigma,\la_1,\lC_1).$$
Hence, e.g., writing $c=c(\sigma)$, we mean that $c$ depends on $\ua,\uC,\la,\lC$.

The  global weak  lower scaling condition (assumption (\ref{eq:assum1})) implies $p_t(x)$ is jointly continuous on $(0,\infty)\times \R$ ($e^{-t\psi}\in L^1(\R)$) and (see \cite[Lemma 1.5]{MR3249349})
\begin{eqnarray}
p_t(x)&\stackrel{C}{\approx}& [V^{-1}(\sqrt{t})]^{-1}\wedge \frac{t}{V^2(|x|)|x|},\quad t>0,\,x\in\R,  \label{p_t_L}\\
\nu(x)&\stackrel{C}{\approx}&\frac{1}{V^2(|x|)|x|},\quad x\neq 0, \label{nu_L}
\end{eqnarray}
where $C=C(\sigma)$.

Let us denote
$$
p(t,x,y)=p_t(y-x)\,.
$$
By \cite[Theorem 1.1 (c)]{2017arXiv171007793G}, we have
\begin{equation}\label{dp_t_L}
|\partial_x p(t,x,y)|\leq c \frac{1}{V^{-1}(\sqrt{t})} p(t,x,y), \quad t>0, x,y\in\R.
\end{equation}

\noindent We consider a compensated potential kernel
  $$K(x)=\int^\infty_0(p_s(0)-p_s(x))ds, \quad x\in\RR.$$ 
By symmetry and \cite[Theorem II.19]{MR1406564},
the monotone convergence theorem implies
$$K(x)=\frac{1}{\pi}\int^\infty_0(1-\cos xs)\frac{1}{\psi(s)}ds=\frac{1}{x\pi}\int^\infty_0(1-\cos s)\frac{1}{\psi(s/x)}ds,\quad x\neq 0.$$
By \cite[Proposition 2.2]{grzywny_ryznar_2016}, $K$ is subadditive. 
\begin{lem}\label{est:gradK}
For every $R > 0$ there exists a constant $\CXV = \CXV(\underline{\sigma}, R\vee1)$ such that 
$$
\vert\partial_x K(x)\vert \leq \CXV M(|x|\wedge R),\,\,\,\,x\in \R.
$$
\end{lem}
\begin{proof}
By symmetry we consider only $x\geq 0$. Let $r = x \wedge R$. Since $x \mapsto p_s(x)$ is nonincreasing on $(0,\infty)$, monotonicity of $V(x)$, \eqref{p_t_L} and \eqref{dp_t_L}  imply

\begin{eqnarray*}
0\leq\partial_x K(x) &=& \partial_x\int\limits_0^\infty \int\limits_0^x \partial_{\rho} p_s(\rho)\,d\rho\,ds = \int\limits_0^\infty |\partial_x p_s(x)|\,ds\\ &\leq& c\int\limits_0^\infty \frac{1}{V^{-1}(\sqrt{s})^2}\wedge\frac{s}{V^2(x)xV^{-1}(\sqrt{s})}\,ds \\
&\leq & c\int\limits_0^\infty \frac{1}{V^{-1}(\sqrt{s})^2}\wedge\frac{s}{V^2(r)rV^{-1}(\sqrt{s})}\,ds \\
&=& c\int\limits_{V^2(r)}^\infty \frac{1}{V^{-1}(\sqrt{s})^2} \,ds + \frac{c}{V^2(r)r}\int\limits_0^{V^2(r)} \frac{s}{V^{-1}(\sqrt{s})}\,ds.
\end{eqnarray*}

By \cite[Remark 4]{MR3165234}, $V^{-1} \in \WLSC{\frac{2}{\ua}}{0}{\uC^{-2/\ua}}$, where $1 < \ua < 2$, hence

$$
\int\limits_{V^2(r)}^\infty \frac{1}{V^{-1}(\sqrt{s})^2} \,ds 
\leq \frac{c_1V(r)^{4/\ua}}{r^2}\int\limits_{V^2(r)}^\infty \frac{1}{s^{2/\ua}}ds = \frac{c_1}{1-\frac{2}{\ua}}\frac{V^2(r)}{r^2}, 
$$
where $c_1 = \uC^{4/\ua}$.
By explanation of \cite[Remark 4]{MR3165234} and \eqref{scalV2}, we have 
\begin{equation}
\frac{V^{-1}(\eta t)}{V^{-1}(t)} \geq c\eta^{\frac{2}{\la_1}}
\end{equation}
for $0 < t < 1$, $\eta < 1$ and some constant $c = c(\la_1, \lC_1, \CI, V(1))$. This implies

\begin{align*}
\int\limits_0^{V^2(r)} \frac{s}{V^{-1}(\sqrt{s})}\,ds &= \int\limits_0^{V^2(r)\wedge 1} \frac{s}{V^{-1}(\sqrt{s})}\,ds + \int\limits_{V^2(r)\wedge 1}^{V^2(r)} \frac{s}{V^{-1}(\sqrt{s})}\,ds \\
&\leq \frac{c_2}{r}\int\limits_0^{V^2(r)\wedge 1} \frac{V(r)^{2/\la_1}}{s^{1/\la_1-1}}\,ds + c_3\frac{V^4(r)}{2V^{-1}(1)} \leq c_4(1+R)\frac{V^4(r)}{r}.
\end{align*}
Hence,
$$
\frac{1}{V^2(r)r}\int\limits_0^{V^2(r)} \frac{s}{V^{-1}(\sqrt{s})}\,ds \leq c_4(1+R)\frac{V^2(r)}{r^2}.
$$
\end{proof}
By \cite[Lemma 2.14 with $\la_1$]{grzywny_ryznar_2016}, for $|x| \leq R$,
\begin{equation}\label{est:KbyV}K(x) \approx \frac{V^2(|x|)}{|x|}.
\end{equation}
Hence,  
\begin{equation}\label{MappK}
|\partial K(x)| \leq M(|x|\wedge R) \approx \frac{K(|x|\wedge R)}{|x|\wedge R}. 
\end{equation}

Analogously to $\alpha$-stable processes, we define the Kato class for gradient perturbations.
\begin{definition}
We say that a function $b \colon \RR \to \RR$ belongs to the Kato class $\pK_1$  if
\begin{equation*}  
\lim_{r\to0^+}\sup_{x\in\RR}\int_{B(x,r)}\frac{K(|x-z|)}{|x-z|} |b(z)|dz=0.
\end{equation*}
\end{definition}
We note that $L^{\infty}(\R)\subset \pK_1$.
Since $\frac{K(r)}{r} \approx \frac{V^2(r)}{r^2}$ for small $r>0$, in this paper we will use the condition \eqref{eq:Kc} in the form
\begin{equation}\label{eq:Kc}
\lim_{r\to0^+}\sup_{x\in\RR}\int_{B(x,r)}\frac{V^2(|x-z|)}{|x-z|^2} |b(z)|dz=0.
\end{equation}

We consider the time-homogeneous transition probabilities
$$
P_t(x,A) =\int_A p(t, x,y)dy, \qquad t>0, x\in \R, A\subset\R.
$$
By Kolmogorov's and Dinkin-Kinney's theorems the transition
probability $P_t$ define in the usual way Markov probability measure 
$\{\PP^x,\,x\in \R\}$ on the space $\Omega$ of the
right-continuous and left-limited functions $\omega :[0,\infty)\to \R$.
We let $\EE^x$ be the corresponding expectations.
We will denote by $X=\{X_t\}_{t\geq 0}$  the canonical process on $\Omega$, $X_t(\omega) = \omega(t)$. Hence,
$$
\PP(X_t \in B) = \int_B p(t,x,y)dy. 
$$

For any open set $D$, we define {\it the first exit time}\/ of the process $X_t$ from $D$,
$$\tau_D=\inf\{t>0: \, X_t\notin D\}\,.$$
Now, by the usual Hunt's formula, we define the transition density of the process {\it killed}\/ when leaving $D$
(\cite{MR0264757}, \cite{MR1329992}, \cite{MR3249349}):
$$
p_D(t,x,y)=p(t,x,y)-\EE^x[\tau_D<t;\, p(t-\tau_D, X_{\tau_D},y)],\quad t>0
,\,x,y\in \R \,.
$$
We briefly recall some well known properties of $p_D$ (see \cite{MR3249349}).
The function $p_D$ satisfies the Chapman-Kolmogorov equations
$$
\int_\R p_D(s,x,z)p_D(t,z,y)dz=p_D(s+t,x,y)\,,\quad s,t>0 ,\,
x,y\in \R\,.
$$
Furthermore, $p_D$ is jointly continuous when $t\neq 0$, and we have
\begin{equation}\label{eq:gg}
  0\leq p_D(t,x,y)=p_D(t,y,x)\leq p(t,x,y)\,.
\end{equation}
In particular,
\begin{equation}
  \label{eq:9.5}
  \int_\R p_D(t,x,y)dy\leq 1\,.
\end{equation}
If $D$ is a $C^{1,1}$ domain (see definition in Section 3), by Blumenthal's 0-1 law, symmetry of $p_t$, we have
$\PP^x(\tau_D=0)=1$ for every $x\in D^c$.
In particular, $p_D(t,x,y)=0$ if $x\in D^c$ or $y\in D^c$.

\section{Green function of $\mathcal{L}$}

We define the Green function of $X_t$ for $D$,
\begin{align}\label{def_g}
 G_D(x,y)=\int_0^\infty p_D(t,x,y)dt,  \qquad x,y \in \R\,
\end{align}
and the Green operator
\begin{align}\label{def_g_op}
 G_D\,\phi(x)=\int_\R G_D(x,y)\phi(y)dy,  \qquad x \in \R\,.
\end{align}

From now on, every time we will mention the Green function, it should be understand as a Green function of $D$, and then $G = G_D$. 

\begin{definition}
We say that a function $f:\R\rightarrow\R$ is a ${\mathcal L}-$harmonic (or simply harmonic) function on an open bounded set $D\subset \R$ if for any open $\overline{F} \subset D$ and $x\in F$ 
$$
f(x) = \E^xf(X_{\tau_F}).
$$
We say that a function $f$ is a regular ${\mathcal L}-$harmonic (or simply regular harmonic) function  on an open bounded set $D\subset \R$  if for every $x \in D$
$$
f(x) = \E^xf(X_{\tau_D}).
$$
\end{definition}

Note that for fixed $x \in D$ the function $G(x, \cdot)$ is harmonic on $D\setminus\{x\}$ and regularly harmonic on $D\setminus\overline{B(x,\varepsilon)}$, where $B(x,\varepsilon)\subset D$. By \cite[Theorem 1.1]{MR2603019}, we know that the function $K(x)$ is harmonic on $\{0\}^c$.

\begin{definition}\label{def:C11}
We call a set $D\subset \R$ a $C^{1,1}$ class set at scale $r>0$ if it is an union of open intervals of length at least $r$ and distanced one from another at least $r$. The number $r_0 = \sup\lbrace r: D$ is at scale $r\rbrace$ is called a localization radius.
\end{definition}

Definition \ref{def:C11} corresponds with the definition of multidimensional $C^{1,1}$ set with localization radius $r_0$. In what follows, we assume that  
$$
D \mbox{ is a } C^{1,1} \mbox{ set with } \diam D < \infty \mbox{ and localization radius } r_0=r_0(D).
$$
Some constants will depend on the ratio $\diam D/r_0$ called the distortion of the set $D$.

\begin{lem}\label{GreenEst1}
There exists a constant  $\CII=\CII(\underline{\sigma},\mathrm{diam}(D)/r_0,1\vee \mathrm{diam}(D))$ such that
\begin{equation}
\label{est:G}
G(x,y)\stackrel{\CII}{\approx} V(\dex)V(\dey)\left(\frac{1}{\sqrt{\dex\dey}}\wedge\frac{1}{|x-y|}\right),\quad x,y\in D.
\end{equation}
\end{lem}
\begin{proof}
Note that (see \cite[Proposition 4.4 and Theorem 4.5]{MR3249349}),
$$
p_D(t,x,y) \approx e^{-2\gamma(D) t}\left(\frac{V(\dex)}{\sqrt{t/2}\wedge V(r_0)}\wedge1\right)\left(\frac{V(\dey)}{\sqrt{t/2}\wedge V(r_0)}\wedge1\right)p(t\wedge V^2(r_0), x, y),
$$
 where $ \frac{1}{8} (\mathrm{diam}(D)/r_0)^2 \leq \gamma(D)V^2(r_0) \leq c(\mathrm{diam}(D)/r_0)^{1/2}$.   
Now, integrating them against time, we get
\begin{align*}
G(x,y)\stackrel{c_1}{\approx} &\left(V(\delta_x)\wedge V(r_0)\right)\left(V(\delta_y)\wedge V(r_0)\right)p(V^2(r_0),x,y)\\&+\int^{V^2(r_0)}_0\left(\frac{V(\delta_x)}{\sqrt{t}}\wedge 1\right)\left(\frac{V(\delta_y)}{\sqrt{t}}\wedge 1\right)p(t,x,y)dt,
\end{align*}
where the comparability constant $c_1$ depends on the scaling characteristics in \eqref{scalV1} and \eqref{scalV2} and a distortion of $D$. Now, by the same calculation as in the proof of \cite[Theorem 7.3 (iii) and Corollary 7.4]{2013arXiv1303.6449C}, we obtain
$$G(x,y)\stackrel{c_2}{\approx} V(\dex)V(\dey)\left(\frac{1}{V^{-1}\left(\sqrt{V(\dex)V(\dey)}\right)}\wedge\frac{1}{|x-y|}\right),$$
where the comparability constant $c_2$ depends on the scaling characteristics in \eqref{scalV1} and \eqref{scalV2}, a distortion of $D$ and $1\vee \mathrm{diam}(D)$. 

Let us consider $x,y\in D$ such that $\frac{1}{V^{-1}\left(\sqrt{V(\dex)V(\dey)}\right)} < \frac{1}{|x-y|}$, this means
\begin{equation}\label{eq:VmVV}
V^2(|x-y|) < V(\dex)V(\dey).
\end{equation} 
Without a loss of generality we may and do assume $\dex \leq \dey$. Then,
$$V^2(|x-y|) < V(\dex)V(\dex + (\dey - \dex)) \leq  V(\dex)V(\dex + |x-y|) \leq  V(\dex)[V(\dex)+V(|x-y|)],$$
which implies $V(|x-y|)\leq 2V(\delta_x)$. By monotonicity and subadditivity of $V$ we obtain that
$$V(\delta_x)\leq V(\delta_y)\leq V(|x-y|)+V(\delta_x)\leq 3V(\delta_x).$$ 
As a consequence of \eqref{scalV1}, we obtain 
\begin{equation}\label{eq:comdxdy}
 \dey \leq (18\CXX/\la)^{1/(2\la)} \dex.
 \end{equation}
Again, by \eqref{scalV1}, we get 
\begin{equation}\label{eq:dxdycom}
V^{-1}\left(\sqrt{V(\dex)V(\dey)}\right) \stackrel{c_3}{\approx} \sqrt{\dex\dey},
\end{equation}
where $c_3=c_3(\la,\lC)$. Now, let 
$$
V^2(|x-y|) \geq V(\dex)V(\dey).
$$
We only need to show that $2|x-y|^2 \geq \dex\dey$. Without the loss of generality we can and do assume $\dex < \dey$. By monotonicity of $V$, $|x-y| \geq \dex$. The case $\dey \leq |x-y|$ is obvious. For $\dex \leq |x-y| < \dey$, we have $\dey \leq |x-y| + \dex \leq 2|x-y|$, which completes the proof. 
\end{proof}

\subsection{Estimates of the Poisson kernel}

If $D$ is $C^{1,1}$, it is known that the harmonic measure of $D$ has a density and we call it the Poisson kernel. By the Ikeda-Watanabe formula \cite{MR0142153} it is equal to
\begin{equation}\label{eq_IW}
P_D(x,z)=\int_DG(x,y)\nu(z-y)dy,\quad x\in D,\,z\in \overline{D}^c.
\end{equation}

\begin{lem}\label{est:poissBall}
Let $R>0$ and $B=B(0,R)$. Then  
$$P_B(x,z) \stackrel{\CIV}{\approx} \frac{V(\dex)}{V(\dez)|x-z|}\left(\frac{V(R)}{V(\dez)}\wedge 1\right),\quad x\in B,\,z\in B^c,$$
where $\CIV=\CIV(\underline{\sigma},1\vee R)>0$. 
\end{lem}
\begin{proof}
By \eqref{eq_IW}, Lemma \ref{GreenEst1} and \eqref{nu_L}, there is $c_1=c_1(\underline{\sigma},1\vee R)$ such that
\begin{align*}
P_B(x,z) & = \int_BG_B(x,y)\nu(|y-z|)dz \\ 
& \stackrel{c_1}{\approx} \int_BV(\dex)V(\dey)\left(\frac{1}{\sqrt{\dex\dey}}\wedge\frac{1}{|x-y|}\right)\frac{dy}{V^2(|z-y|)|z-y|}.
\end{align*}

By Remark \ref{remScalExp}, we obtain inequality \eqref{scalV2} for $r<3R$ with constant $c_2=c_2(\la_1,\lC_1,1\vee R)$. Hence, for $|z|<2R$, we have
\begin{align*}
(2\uC\CXX)^{-1/2}\left(\frac{\dez}{|z-y|}\right)^{\ua/2}\leq \frac{V(\dez)}{V(|z-y|)} &\le c_2    \left(\frac{\dez}{|z-y|}\right)^{\la_1/2},\\   (2\uC\CXX)^{-1/2}\left(\frac{\dey}{|z-y|}\right)^{\ua/2}\leq \frac{V(\dey)}{V(|z-y|)} &\le c_2    \left(\frac{\dey}{|z-y|}\right)^{\la_1/2}. 
\end{align*}
These imply
\begin{align*}
P_B(x,z) 
& \le c_1c_2^2 \int_B \frac{V(\dex)}{V(\dez)}\left(\frac{1}{\sqrt{\dex\dey}}\wedge\frac{1}{|x-y|}\right)\frac{(\dey\dez)^{\la_1/2}}{|z-y|^{1+\la_1}} dy \\
& \stackrel{c_3}{\approx}  \int_B \frac{V(\dex)}{V(\dez)}\left(\frac{\dez}{\dex}\right)^{\la_1/2} G_B^{S\la_1 S}(x,y)\frac{dy}{|z-y|^{1+\la_1}} \\
& \approx  \frac{V(\dex)}{V(\dez)}\left(\frac{\dez}{\dex}\right)^{\la_1/2} P_B^{S\la_1 S}(x,z).
\end{align*}
Here, $c_3=c_3(c_1,c_2,\la_1)$ and $S\la_1 S$ refers to the symmetric $\alpha$-stable process with index of stability $\la_1$.  Similarly, we obtain
$$P_B(x,z)\geq c_4 \frac{V(\dex)}{V(\dez)}\left(\frac{\dez}{\dex}\right)^{\ua/2} P_B^{S\ua S}(x,z),$$
where $c_4=c_4(c_1,\ua,\uC)$.
By formula for $P_B^{S\la_1 S}(x,z)$ \cite[Theorem A]{MR0126885}, we get the assertion of the lemma for $|z|<2R$. 

If $|z|\geq 2R$, by \eqref{subadd} and \cite[Proposition 3.5]{MR3007664}, we get
$$P_B(x,z)\approx \nu(|z|)\EE^x\tau_B\approx V(\dex)V(R)\nu(|z|),$$
which implies the claim of the lemma.

\end{proof}
\begin{prop}\label{PforD}
There exists a constant $\CV =\CV(\underline{\sigma},\mathrm{diam(D)}/r_0,1\vee \mathrm{diam(D)})$ such that 
$$
P_D(x,z)\stackrel{\CV}{\approx}\frac{V(\dex)}{V(\dez)|x-z|}\left(\frac{V(\mathrm{diam(D)})}{V(\dez)}\wedge 1\right), \quad x\in D,\,z\in D^c.
$$

\end{prop}
\begin{proof}
Let $x \in D$, $z \in D^c$. By Lemma \ref{est:poissBall} we consider only the case when $D$ is a sum of at least two open intervals. Let $B$ be an open interval such that $x\in B$ and $\tilde{D}=D\setminus B$ is open.
By the Ikeda-Watanabe formula
\begin{equation}\label{eq:general_D}
P_D(x,z) =\int_B G(x,y)\nu(|y-z|)dy + \int_{\tilde{D}}G(x,y)\nu(|y-z|)dy=:\mathrm{I}+\mathrm{II}.
\end{equation}
Lemma \ref{est:poissBall} implies 
$$G(x,y)\stackrel{c_1}{\approx} G_B(x,y),\quad x,y\in B,$$
for $c_1=c_1(C_2)$. Hence, by Lemma \ref{est:poissBall}
\begin{equation}\label{eq:Poi_gen1}
\mathrm{I} \stackrel{c_1}{\approx} \int\limits_B G_B(x,y)\nu(|y-z|)dy \stackrel{c_2}{\approx}\frac{V(\dex)}{V(\mathrm{dist}(z,B))|x-z|}\left(\frac{V(\mathrm{diam}(B))}{V(\dez)}\wedge 1\right),
\end{equation}
where $c_2=c_2(C_2,C_3)$. If $\mathrm{dist}(z,B)=\dez$,  the lower bound follows by \eqref{eq:Poi_gen1}. Suppose $\mathrm{dist}(z,B)<\dez$ and let $\tilde{B}$ be a connected component of $\tilde{D}$ such that  $\mathrm{dist}(z,\tilde{B})=\dez$. Therefore, by Lemma \ref{GreenEst1},
$$\mathrm{II}\geq C_2\int_{\tilde{B}}\frac{V(\dex)V(\dey)}{\textrm{diam}(D)}\nu(|y-z|)dy\geq \frac{C_2 r_0}{2\mathrm{diam}(D)}\frac{V(\dex)}{|x-z|}\int_{\tilde{B}}V(\dey)\nu(|y-z|)dy.$$
Now, \eqref{nu_L} and \eqref{subadd} imply
\begin{align*}\int_{\tilde{B}}V(\dey)\nu(|y-z|)dy&\geq  \frac{c_4}{V^2(2\dez)\dez}\int^{\dez\wedge r_0/2}_{0}V(s)ds\geq  \frac{c_4}{V^2(2\dez)\dez}\int^{\dez\wedge r_0/2}_{(\dez\wedge r_0/2)/2}V(s)ds \\&\geq \frac{c_4}{4} \frac{(\dez\wedge r_0/2)V(\dez\wedge r_0/2)}{V^2(2\dez)\dez}.
\end{align*}
Hence, we obtain the lower bound in this case.

Next, we will prove the upper bound for the second integral. Let $\lambda= \dez\wedge \mathrm{diam}(D)$ and $D_1 = \tilde{D}\cap\{y: \dey \leq \lambda\}$ and $D_2 = \tilde{D}\cap\{y: \dey > \lambda\}$. By  weak scaling conditions, we obtain
\begin{eqnarray*}
\mathrm{II}&\stackrel{c_5}{\approx}&\int\limits_{\tilde{D}}\frac{V(\dex)V(\dey)}{|x-y|}\frac{dy}{V^2(|z-y|)|z-y|}\leq  \frac{V(\dex)}{r_0}\int\limits_{\tilde{D}}\frac{V(\dey)dy}{V^2(|z-y|)|z-y|} \\
&\leq & \frac{V(\dex)}{r_0}\left(\frac{|D_1|V(\lambda)}{V^2(\dez)\dez}
+\int\limits_{D_2}\frac{dy}{V(\dey)\dey}\right) \leq  \frac{V(\dex)}{r_0}\left(\frac{2\mathrm{diam}(D)}{r_0}\frac{\lambda V(\lambda)}{V^2(\dez)\dez}
+c_6 \frac{\textbf{1}_{\lambda=\dez}}{V(\lambda)}\right)\\&\approx& \frac{V(\dex)}{V(\dez)|x-z|}\left(\frac{V(\diam{D})}{V(\dez)}\wedge1\right), 
\end{eqnarray*}
where $c_5=c_5(C_2,C)$ and $c_6$ depends only on the scaling characteristics.
This completes the proof.
\end{proof}

\subsection{Estimates of $\partial_x G(x,y)$}
Below, we will prove various estimates of $\partial_x G(x,y)$ according to the range of variables $x$ and $y$. We summarize these results in  Theorem \ref{thm:GradEst}.
First, we will need the following auxiliary lemma.
\begin{lemma}\label{lem:M/V}
Let $x \in D$. There is a constant $\CVI=\CVI(\underline{\sigma},\mathrm{diam(D)}/r_0,1\vee \mathrm{diam(D)})$ such that
\begin{align*}
\int_{\RR} \frac{M(|x-z|)}{V(\dez)} dz \le \CVI \frac{V(\dex)}{\dex}.
\end{align*}
\end{lemma}

\begin{proof}
Let $B_1 = B(x,\dex/2)$. By \eqref{scalV2R}, we have
\begin{align*}
\int_{B_1} \frac{M(|x-z|)}{V(\dez)} dz &\le 2 \int_{B_1} \frac{V^2(|x-z|)}{|x-z|^2 V(\dex)} dz \le c_1 \int_{B_1} \frac{V^2(\dex) |x-z|^{\la_1-2}}{\dex^{\la_1}V(\dex)} dz \le c_2 \frac{V(\dex)}{\dex}.
\end{align*}
Note that for $z \not\in B$,  we have $\dez \le 3|x-z|$ and $\dex \le 2|x-z|$. Hence, by \eqref{scalV1},
\begin{align*}
\int_{B_1^c \cap \{\dex \le \dez\}} \frac{M(|x-z|)}{V(\dez)} dz &\le c_3 \int_{\{\dex \le \dez\}} \frac{V(\dez)}{\dez^2 } dz \le c_4 \int_{\{\dex \le \dez\}} \frac{V(\dex) \dez^{\ua/2-2}}{\dex^{\ua/2}} dz \le c_5 \frac{V(\dex)}{\dex}, \\
\int_{B_1^c \cap \{\dez < \dex\}} \frac{M(|x-z|)}{V(\dez)} dz &\le c_6 \int_{\{\dez < \dex\}} \frac{V^2(\dex)}{\dex^2 V(\dez)} dz \le c_7 \int_{\{\dez < \dex\}} \frac{V(\dex)}{\dex^{2-\ua}\dez^{\ua}} dz  \le c_8\frac{V(\dex)}{\dex}.
\end{align*}
\end{proof}

\begin{prop}\label{prop:estGK}
There is a constant $\CVII=\CVII(\underline{\sigma},\mathrm{diam(D)}/r_0,1\vee \mathrm{diam(D)})$ 
such that
$$|\partial_xG(x,y)|\leq \CVII\left( M(|x-y|) + \frac{G(x,y)}{\dex}\mathbbm{1}_{\frac{|x-y|}{2} > \dex}\right)$$\end{prop}
\begin{proof}
Since $X_t$ is translation invariant, we may and do assume that $0\notin D$. Let $x, y \in D$ and $x\neq y$. It is known (see \cite[Lemma 2.3]{grzywny_ryznar_2016})
\begin{equation}\label{G0}
G_{\{0\}^c}(x,y)=K(x)+K(y)-K(y-x).
\end{equation} 
Hence, by symmetry,
\begin{eqnarray*}
G(x,y)&=&G_{\{0\}^c}(x,y)-\EE^y G_{\{0\}^c}(x,X_{\tau_D})\\
&=&K(y)-K(x-y)-\EE^y K(X_{\tau_D})+\EE^yK(x-X_{\tau_D}).
\end{eqnarray*}
By Lemma \ref{est:gradK} and  
the dominated convergence theorem,
\begin{eqnarray*}
\partial_x G(x,y)&=& \EE^y \partial_x K(x-X_{\tau_D})-\partial_x K(x-y).
\end{eqnarray*}
Again, by Lemma \ref{est:gradK} and \eqref{MappK}, for $ |x-z| \geq |x-y|/2$, we have 
\begin{align*}
|\partial_x K (x- z)| \leq c_1 M(|x - z|\wedge \diam(D)) \leq c_2 M(|x - y|). 
\end{align*}
This implies 
\begin{eqnarray*}
|\partial_x G(x,y)|&\leq& c_3M(|x-y|) + \EE^y| \partial K(x-X_{\tau_D})|\\
& \leq & c_4M(|x-y|) + \EE^y\left[ |\partial K(x-X_{\tau_D})|, |x-X_{\tau_D}| \leq \frac{|x-y|}{2}\right].
\end{eqnarray*}
It remains to estimate
\begin{equation}\label{int:cutGradK}
I:=\EE^y \left[|\partial K(x-X_{\tau_D})|, |x-X_{\tau_D}| \leq \frac{|x-y|}{2}\right].
\end{equation}
If $\dex \geq \frac{|x-y|}{2}$, $I = 0$. 
So let $\dex < \frac{|x-y|}{2}$.  Note that if $|x-z| \leq |x-y|/2$, then $|y-z| \geq |x-y|/2$, and in consequence, by Proposition \ref{PforD},
$$
P_D(y,z) \lesssim \frac{V(\dey)}{V(\dez)}\frac{1}{|y-z|} \leq \frac{V(\dey)}{V(\dez)}\frac{2}{|x-y|}.
$$
By Lemma  \ref{lem:M/V},
\begin{align*}
I 
\leq \int_{D^c\cap B(x, \frac{|x-y|}{2})}M(|x-z|)P_D(y, z)dz
& \leq c_5\frac{V(\dey)}{|x-y|}\int_{\RR} M(|x-z|)\frac{dz}{V(\dez)} \\ & \leq  c_6\frac{V(\dey)}{|x-y|}\frac{V(\dex)}{\dex}.
\end{align*}
Since $\dex \leq \frac{|x-y|}{2}$, we have $\dey \leq \frac{3}{2}|x-y|$ and
$$
\frac{V(\dey)}{|x-y|}\frac{V(\dex)}{\dex} \lesssim\frac{ G(x,y)}{\dex}.
$$
Hence, 
$$
|\partial_xG(x,y)|\lesssim M(|x-y|) + \frac{G(x,y)}{\dex}\mathbbm{1}_{\frac{|x-y|}{2} > \dex},
$$
which ends the proof.
\end{proof}

By Lemma \ref{GreenEst1} and Proposition \ref{prop:estGK}, we get  a weaker but also useful estimate. 
\begin{cor}\label{cor:estGradGreena}
There is a constant $\CVIII=\CVIII (\underline\sigma, \diam(D)\vee 1)$ such that   for any open $D\neq \R$
\begin{equation}\label{estGradGreena} 
|\partial_xG(x,y)| \leq \CVIII M(\dex\wedge|x-y|).
\end{equation}
\end{cor}

\begin{lem}\label{GradGreen}           
  If $f \in \mathcal{K}_1$, then
  \begin{equation}\label{eq:GradGreen}
    \partial_y\int_D  G(y,z)f(z)\,dz = \int_D \partial_y\, G(y,z)f(z)\,dz\,,\quad y \in D\,.
  \end{equation}
\end{lem}
\begin{proof}
 Let $0<h<\dey/2$.  
Then,
\begin{eqnarray*}
\left|\frac{G(y+h,z)- G(y,z)}{h}\right|&=&\frac{1}{h}\left|\int\limits_0^1\partial_sG(y+sh, z)ds\right|\\
= \left|\int\limits_0^1\partial_yG(y+sh, z)ds\right|&\leq & \CVIII \int\limits_0^1\left( M(\delta_{y+sh}\wedge|y+sh-z|) \right)\,ds \\
&\leq &\CVIII  \int\limits_0^1\left( M(\dey / 2) + M(|y+sh-z|)\right) ds.
\end{eqnarray*}

Since $f\in\mathcal{K}_1$ and the integrand is uniformly in $h$ integrable on and $(0,1)\times D$, which ends the proof. 
\end{proof}

\begin{prop}\label{gradParGreen}
Let $x \in D$, $0 < \varepsilon < \dex$, $B = B(x, \varepsilon)$ and $A = B^c\cap D$. Then,
$$
\partial_x G(x,y)= \int_B \partial_xG(x,z)P_{A}(y,z)dz.
$$
\end{prop}
\begin{proof}
Fix $x \in D$. Then, $G(x, \cdot )$ is regular harmonic on $A = D \cap [ x - \varepsilon, x + \varepsilon]^c$ for every $0 < \varepsilon < \dex$. This means
$$
G(x,y) = \EE^yG(x,X_{\tau_A}) = \int_B G(x,z)P_{A}(y,z)dz,\,\,\,\,\, y \in A.
$$
Let us fix  $y\in A$. For $z \in B$, we define 
$P_1(y, z) = P_{A}(y, z)  \textbf{1}_{B(x, \varepsilon/2)}(z)$ and $P_2(y,z) = P_A(y,z) - P_1(y,z)$. Since $P_1$ is bounded, we have $P_1 \in \pK_1$ and by Lemma \ref{GradGreen}, 
$$
\partial_x\int_B G(x,z)P_1(y,z)dz = \int_B \partial_x G(x,z) P_1(y,z)dz. 
$$

Since $\partial_x G(x,z)$ is finite on the support of $P_2(y, \cdot)$,  by the mean value theorem and the dominated convergence theorem, we get

\begin{align*}
&\lim\limits_{h \rightarrow 0} \frac{\int_B G(x + h,z)P_2(y,z)dz - \int_B G(x,z)P_2(y,z)dz}{h} 
 \\ &= \lim\limits_{h \rightarrow 0} \int_B\frac{ G(x + h,z)- G(x,z)}{h}P_2(y,z)dz = \int_B \partial_x G(x,z)P_2(y,z)dz.
\end{align*}
These imply
\begin{align*}
\partial_x\int_B G(x,z)P_{B^c}(y,z)dz &= \partial_x\int_B G(x,z)P_1(y,z)dz + \partial_x\int_B G(x,z)P_2(y,z)dz \\
& = \int_B \partial_x G(x,z) P_1(y,z)dz + \int_B \partial_x G(x,z)P_2(y,z)dz \\
& = \int_B \partial_xG(x,z)P_{A}(y,z)dz,
\end{align*}
which completes the proof. 
\end{proof}

\begin{lem}\label{ineq:dgldx}
Let $x, y \in D$ and $\dex < 2|x-y|$. Then, there exists a constant $\CIX=\CIX(\underline{\sigma},\mathrm{diam(D)}/r_0,1\vee \mathrm{diam(D)})$  such that
$$
|\partial_xG(x,y)|\leq \CIX\frac{G(x,y)}{\dex}.
$$
\end{lem}
\begin{proof}
Let $B \subset D$ be any interval such that $\overline{B}\subset D$ and put $A = B^c\cap D$. For any $x \in B$ and $y \in D$ such that $x\neq y$, by Propositions \ref{gradParGreen} and \ref{prop:estGK} and harmonicity of $G$,
\begin{eqnarray*}
|\partial_x G(x,y)| &= &|\int_{B}\partial_x G(x,z)P_A(y, z)dz| \\
& \leq & \CVII\int_{B} \left( M(|x-z|) + \frac{G(x,z)}{\dex}\mathbbm{1}_{\frac{|x-z|}{2} > \dex}\right) P_A(y, z)dz\\
& \leq & \CVII\int_{B}  M(|x-z|) P_A(y, z)dz + \CVII\frac{G(x,y)}{\dex}.
\end{eqnarray*}
Therefore, it remains to estimate the integral 
\begin{equation}\label{introB}
\int_{B} M(|x-z|) P_A(y, z)dz.
\end{equation}

Let $B = B(x,\dex/4)$. By the assumption $y \not\in B$, $\dist(y, B) \ge \dex/4$ and $|y-z| \approx |x-y|$ for $z \in B$. 
Denote $\delta_{x}^A = \dist(x,\partial A)$. Note that $\dex^A \approx \dex$  and $\dey^A \approx \dey$. By Proposition \ref{PforD} and Lemmas \ref{lem:M/V}, \ref{GreenEst1}, we get
\begin{align*}
& \int_{{B}} M(|x-z|)P_A(y, z)dz \leq \CIV \int_{{B}} M(|x-z|)\frac{V(\dey^A)}{V(\dez^A)|y-z|}dz \\ & \leq  c_1\frac{V(\dey)}{|x-y|}\int_{{B}} \frac{M(|x-z|)}{V(\dez^A)}dz \leq  c_2\frac{V(\dey)V(\dex^A)}{|x-y|\dex^A} \le c_3 \frac{V(\dey)V(\dex)}{|x-y|\dex} \leq c_4 \frac{G(x,y)}{\dex}\,.
\end{align*}

Since constants $c_1-c_4$ depend on $D$ only via constants $\CXV$, $\CIV$ and $\CVI$, the proof is completed. 
\end{proof}

\begin{thm}\label{thm:GradEst}
There is a constant $\CX=\CX(\underline{\sigma},\mathrm{diam}(D)/r_0,1\vee\mathrm{diam}(D))$ such that
$$|\partial_x G(x,y)|\leq \CX\frac{G(x,y)\wedge K(|x-z|)}{|x-y|\wedge\dex},\quad x,y\in D.$$
\end{thm}
\begin{proof}
Due to Corollary \ref{cor:estGradGreena}, Lemma \ref{ineq:dgldx}, \eqref{est:KbyV} and \eqref{MappK} it remains to prove existing of a constant $c$ such that 
$$
\frac{G(x,y)}{|x-y|}\geq c M(|x-y|),
$$
when $|x-y|\leq\dex/2$. But in this case $\dex\approx\dey$ and therefore, by Lemma \ref{GreenEst1}, 
$$G(x,y)\stackrel{C_2}{\approx}\frac{V^2(\dex)}{\dex}.$$
Since $\la_1>1$, by  \eqref{scalV2}, we obtain that $s\mapsto V^2(s)/s$ is almost increasing (bounded from below by an increasing function). Hence, we get the claim.
\end{proof}

We end this section we the proof  of the uniform intergability of $\partial_z G(z,y)$.

\begin{lem}\label{GDUnif} The function $\partial_z G(z, y)$ is uniformly in $y$ integrable against $|b(z)|dz$.
\end{lem}

\begin{proof}

It is enough to show that
\begin{equation*}
\lim_{N\to \infty}\sup_{y\in\R}\int\limits_{|\partial_z G(z, y)| > N} |\partial_z G(z,y)||b(z)|dz = 0.
\end{equation*}
Let $N>0$ and $r_N = \inf \lbrace r >0  \colon M(r) \le N/\CVIII \rbrace \land r_0$. Note that $\lim_{r\to 0} M(r) =\infty$, hence $r_N \to 0$ as $N \to \infty$.  Fix $y \in \R$ and take $N$ such that $r_N \le r_0$. By \eqref{estGradGreena},  $\lbrace z : |\partial_z G(z, y)| > N\rbrace  \subset \lbrace z : M(\dez) > N/\CVIII \rbrace \cup \lbrace z : M(|z-y|) > N/\CVIII\rbrace \subset  \lbrace z : \dez < r_N \rbrace \cup \lbrace z : |y-z| < r_N \rbrace$. We may assume that the set $D$ is an union of $k$ distinctive intervals. By Proposition \ref{prop:estGK} and monotonicity of $M(\cdot)$, we have

\begin{align}
& \int\limits_{|\partial_z G(z, y)| > N} |\partial_z G(z, y)||b(z)|\,dz 
\\&\leq \CVIII\left(
\int\limits_{\dez < r_N}M(\dez)|b(z)|\,dz + 
\int\limits_{|z-y| < r_N}M(|z-y|)|b(z)|\,dz
\right) \nonumber 
\leq (2k+1)\CVIII K_{r_N},\label{unifforM}
\end{align}
where
$$K_r = \sup_{y\in\R}\int\limits_{B(y,r)}M(|y-z|)|b(z)|\,dz.$$
By \eqref{eq:Kc}, $\lim\limits_{N\to \infty}K_{r_N} = 0$, which completes the proof. 
\end{proof}

\subsection{3G  inequalities}

Now, we apply the estimates of the Green function and its derivative to obtain the following $3G$-type inequalities.

\begin{prop}\label{lem:3G}
There is a constant $\CXI=\CXI(\underline{\sigma}, 1\vee \mathrm{diam(D)})$ such that
$$
\frac{ G(x,z) G(z,y)}{ G(x,y)}\leq \CXI V(\dez) \(\frac{ G(x,z)}{V(\dex)}\vee \frac{ G(z,y)}{V(\dey)}\right).
$$
\end{prop}
\begin{proof}
For $x, y \in D$, we define
$$
\CG(x,y) = \frac{G(x,y)}{V(\dex)V(\dey)}.
$$
It suffices to prove that for any $x,y,z \in D$, we have
$$
\CG(x,z)\wedge\CG(z,y)\leq c\,\CG(x,y).
$$

By Lemma \ref{GreenEst1},
\begin{align*}
\CII^{-1} \left( \CG(x,z)\wedge\CG(z,y) \right)&\leq \frac{1}{\(\dex\dez\r)^{1/2}}\wedge\frac{1}{|x-z|}\wedge\frac{1}{\(\dez\dey\r)^{1/2}}\wedge\frac{1}{|z-y|}\\
&=\frac{1}{\(\dex\dey\r)^{1/2}}\(\frac{\dex\wedge\dey}{\dez}\r)^{1/2}\wedge\frac{1}{|x-z|\vee|z-y|}
\end{align*}
If $\frac{\dex\wedge\dey}{\dez}\leq 2$, Lemma \ref{GreenEst1} imply
$$ \CG(x,z)\wedge\CG(z,y) \leq 2\CII \left(\frac{1}{\(\dex\dey\r)^{1/2}}\wedge\frac{1}{|x-y|}\right)\leq 2\CII^2\CG(x,y).$$

If $\frac{\dex\wedge\dey}{\dez}\geq 2$, then 
$\sqrt{\dex\dey}\leq $
$\dex\vee\dey\leq 2\left(|x-z|\vee|y-z|\right)$ and in consequence 
\begin{align*}
\CG(x,z)&\wedge\CG(z,y)\leq  \frac{\CII}{|x-z|\vee|z-y|}\leq 2\CII^2\CG(x,y).
\end{align*}
\end{proof}

\begin{lem}\label{est:GGdG}
There is a constant $\CXII=\CXII(\underline{\sigma},\mathrm{diam(D)}/r_0,1\vee \mathrm{diam(D)})$ such that for any $x,y,z\in D$, we have
$$
\frac{G(x,z)|\partial_zG(z,y)|}{G(x,y)} \leq \CXII M(\dez \wedge |y-z|).
$$
\end{lem}
\begin{proof}
Note that $\dez^2 \le 4(\de_x\de_z \vee |x-z|^2) $, hence, 
\begin{equation}\label{est:gGg}
\frac{V(\dez)G(x,z)}{V(\dex)} \approx V(\dez)^2\(\frac{1}{(\de_x\de_z)^{1/2}}\wedge\frac{1}{|x-z|}\) \leq 2\frac{V(\dez)^2}{\de_z}.
\end{equation}
By Proposition \ref{lem:3G} and \eqref{est:gGg},
\begin{align}\label{ineq:3G}
    	\frac{G(x,z)G(z,y)}{G(x,y)} 
    	 \leq   \CXI\(\frac{V(\dez)G(x,z)}{V(\dex)}\right)\vee\(\frac{V(\dez)G(z,y)}{V(\dey)}\right)  
    	\leq  c_1\frac{V(\dez)^2}{\de_z}, 
\end{align}
where $c_1 = 2\CII\CXI$. 
For $\dez<2|y-z|$, by Lemma \ref{ineq:dgldx} and \eqref{ineq:3G}, we get
$$
\frac{G(x,z)|\partial_zG(z,y)|}{G(x,y)} \leq \CIX\frac{G(x,z)}{G(x,y)}\frac{G(z,y)}{\dez}\leq c_2 M(\dez), 
$$
where $c_2 = c_1\CIX$. 
Now, let $\dez\ge2|z-y|$. Note  that $\dez \approx \dey$ and in consequence $G(z,y) \approx V^2(\dez)/\dez$. Hence, by (\ref{estGradGreena}) and (\ref{ineq:3G}), we have
\begin{align*}
  \frac{G(x,z)|\partial_zG(z,y)|}{G(x,y)} &\leq \CVIII \frac{G(x,z)}{G(x,y)}M(|z-y|) \leq \frac{c_3V^2(\dez)}{G(z,y)\dez}M(|z-y|) \leq c_4M(|z-y|),
\end{align*}
where $c_3=c_1\CVIII$ and $c_4=c_3\left(\frac{2\CI}{\underline{c}}\right)\sqrt{\frac{3}{2}}$. Now,	 by \eqref{subadd}, the assertion of the lemma holds.
\end{proof}

For $x, y \in D$, we define
\begin{equation}\label{def:kappa}
\kappa(x,y) = \int\limits_D\left|b(z)\frac{ G(x,z) \partial_zG(z,y)}{G(x,y)}\right|dz,
\end{equation}
\begin{lem}\label{lem:boundKappa}
Let $\lambda <\infty, R < 1$. There is a constant 
$\CXIII = \CXIII(\underline{\sigma}, b, \lambda, R)$  
such that if $\diam(D)/r_0(D) \leq \lambda$ and $\diam(D)\leq R$, then 
\begin{equation}\label{eq:kappa}
\kappa(x,y) \leq \CXIII, \qquad x, y \in D.
\end{equation}
Furthermore, 
$\CXIII \rightarrow 0$ as $R\rightarrow 0$. 
\end{lem}
\begin{proof}
Since $b \in \mathcal{K}_1$, \eqref{eq:kappa} follows by Lemma \ref{est:GGdG}.
\end{proof}

\section{Green function of $\tL$}

Following \cite{MR2283957} and \cite{MR2876511}
we recursively define,
for $t>0$ and $x,y \in \R$, 
\begin{align*}
p_0(t,x,y)  &=  p(t,x,y)\,, \\ 
 p_n(t,x,y)  &=  \int_0^t \int_{\R} p_{n-1}(t-s,x,z) b(z) \partial_z p(s,z,y)\,dz\,ds\,,\quad n \ge 1\,,
\end{align*}
and we let
\begin{equation}
\tp(t,x,y)=\sum_{n=0}^\infty p_n(t,x,y)\,.
\end{equation}
By \cite[Theorem 1.1]{MR2876511} the series converges absolutely, $\tp$ is a continuous  probability transition density function,  and
\begin{equation}\label{ptxy_comp}
  c_T^{-1} p(t,x,y) \le \tp(t,x,y) \le c_Tp(t,x,y)\,,\qquad x,y \in \R\,,\; 0<t<T\,,
\end{equation}
where $c_T \to 1$ if $T \to 0$, see \cite[Theorem 2]{MR2283957}.

By Chapman-Kolmogorov equation, there is $\CXIV>0$ such that
\begin{equation}
\label{hkfreecomp}\CXIV^{-1-t} p(t,x-y)\leq \tp(t,x,y)\leq \CXIV^{t+1} p(t,x-y), \quad  t>0,\, x,y\in \RR.
\end{equation}

We let $\tPP$, $\tEE$ be the Markov distributions and expectations defined by transition density $\tp$ on the canonical path space.
By Hunt formula,
\begin{equation}\label{eq:Hunt}
\tp_D(t,x,y) = \tp(t,x,y) - \tEE^x\left[\tau_D < t;\; \tp(t-\tau_D,  X_{\tau_D},y)\right]\,.
\end{equation}

Except symmetry, $\tp_D$ has analogous properties as $p_D$, i.e. the Chapman-Kolmogorov equation holds
$$
\int_{\RR^d} \tp_D(s,x,z)\tp_D(t,z,y)dz=\tp_D(s+t,x,y)\,,\quad s,t>0 ,\,
x,y\in \R,
$$
$0\leq \tp_D(t,x,y)\leq \tp(t,x,y)$ and $\tp_D$ is jointly continuous on $(0,\infty)\times D \times D$.

We denote by $\tG_D(x,y)$ the Green function 
of $\tilde{\mathcal{L}} = \mathcal{L}+b \partial$ on $D$,
\begin{align}\label{eq:deftG}
	\tG_D(x,y)&=\int_0^\infty \tp_D(t,x,y)dt\,.
\end{align}
As for $G$, from now on, every time we will mention the Green function $\tG$, it should be understand as a Green function of $\tilde{\mathcal{L}}$ on $D$, and then $\tG = \tG_D$.

By Blumenthal's 0-1 law and \eqref{hkfreecomp},  $\tp_D(t,x,y)=0$ and $\tG(x,y)=0$ if $x\in D^c$ or $y\in D^c$.
By (\ref{ptxy_comp}), we have 
$$
\lim_{t \to 0} \frac{\tp(t,x,y)}{t} = \lim_{t \to 0}
\frac{p(t,x,y)}{t} = \nu(y-x)\,.
$$
Thus the intensity of jumps of the canonical process $X_t$ under $\tilde{\PP}^x$ is the same
as under $\PP^x$. Accordingly, we obtain the following description.

\begin{lem}\label{lem:lsgp}
The $\tPP^x$-distribution of $(\tau_D,X_{\tau_D})$ on $(0,\infty)\times (\overline{D})^c$ has density
\begin{equation}
  \label{eq:IWft}
  \int_D \tp_D(u,x,y)\nu(z-y)\,dy\,,\quad u>0\,,\;\dez>0\,.
\end{equation}
\end{lem}
We define the Poisson kernel of $D$ for $\tilde{\mathcal{L}}$,
\begin{equation}\label{eq:djpt}
\tilde{P}_D(x,y)=\int_D \tG(x,z)\nu(|y-z|)\,dz\,,\quad x\in D\,,\;y\in D^c\,.
\end{equation}
By (\ref{eq:deftG}), (\ref{eq:djpt}) and (\ref{eq:IWft})  we have 
\begin{equation}\label{eq:IWt}
\tPP^x(X_{\tau_D} \in A) =  \int_A\int_D \tG(x,z)\nu(|y-z|)\,dz\,dy =\int_A \tP_D(x,y)dy\,,
\end{equation}
if $A \subset (\bar{D})^c$. For the case of $A\subset \partial D$, we refer the reader to Lemma~\ref{l:nu}.

\begin{lem}\label{lem:GLDupper}
$\tG(x,y)$ is continuous and 
$$
\tG(x,y)\leq \CXVII, \qquad x,y\in \mathbb{R},
$$
where $\CXVII = \CXVII(\sigma, b, \diam(D))$.
\end{lem}
\begin{proof}
In the same way as in \cite[Lemma 7]{MR2892584} we get that there are constants $c$ and $C$ such that
\begin{equation}\label{eq:eotp}
\tp_D(t,x,y) \le Ce^{-ct}\,,\quad t>1, \quad x,y \in
  \RR\,.
\end{equation}

By (\ref{eq:deftG}),   (\ref{ptxy_comp}) and (\ref{eq:eotp}) we obtain
  \begin{align*}
    \tG(x,y)
    & \le \int_0^1 \CXIV p(t,x,y)\,dt + \int_1^\infty Ce^{-ct}\, dt \\
    & \le \int_0^1 p(t,0,0)\,dt + \int_1^\infty Ce^{-ct}\, dt \\
    & \le c_1 + C/c,
  \end{align*}
where $c_1$ is finite bound for $\int_0^1p(t,0,0)\,dt$. We put $\CXVII = c_1+C/c$. By (\ref{eq:deftG}), continuity of $\tilde{p}_D$ and the dominated convergence theorem, $\tG(x,y)$ is continuous. 
\end{proof}

By Lemmas \ref{lem:GLDupper} and \ref{GDUnif}, for every $x\in D$, the  function  
$$
f_x(y) := \tG(x,y)-G(x,y)-\int_D\tG(x,z)b(z)\partial_z G(z,y)dz
$$ 
is well defined, integrable and bounded on $\mathbb{R}$. 
Hence, following \cite[Theorem 3.1]{2017-TG-TJ-GZ-pms}, we obtain the following perturbation formula (for the proof see \cite{2017-TG-TJ-GZ-pms}).
\begin{lem}\label{lem:pf}
Let $x,y\in\mathbb{R}$. We have
\begin{equation}\label{eq:wp}
\tG(x,y) = G(x,y)+\int_D\tG(x,z)b(z)\partial_zG(z,y)\,dz.
\end{equation}
\end{lem}

\section[Local results] {Proof of Theorem \ref{Theorem1}}
\label{chap:Green}

First, we will  prove the comparability of $G$ and $\tG$ for small sets $D$ from the $C^{1,1}$ class. 
For this purpose we could consider the perturbed series for $\tG$ as it was presented in \cite{2017-TG-TJ-GZ-pms}. We could define by induction the functions $\kk_n$ and show the convergence and estimates of the series 
$$
\tG(x,y) = \sum\limits_{n=0}^{\infty}\kk_n(x,y).
$$
However, since $\tilde{G}$ is bounded, we present a simpler proof of the following lemma (compare \cite[Lemma 3.11]{2017-TG-TJ-GZ-pms}).

\begin{lem}\label{Theorem1s} 
Let $b\in \pK_1$ and $\lambda>0$. There is $\varepsilon=\varepsilon(\underline{\sigma},b,\lambda)>0$ such that if 
${\rm diam}(D)/r_0(D)\le \lambda$ and $\diam(D)\le\varepsilon$, then
  \begin{equation}
    \label{eq:egfs}
\frac{1}{2}G(x,y) \le \tG(x,y) \le \frac{3}{2} G(x,y), \qquad x,y \in
\R\,.
  \end{equation}
\end{lem}

\begin{proof}
By  Lemma \ref{lem:boundKappa}, there exists $\varepsilon_1 > 0$ such that if $\diam(D) < \varepsilon_1$, then

\begin{align}\label{eq1:Theorem1s}
	\int_DG(x,z)|\partial_zG(z,y)b(z)|dz\leq \CXIII G(x,y),
\end{align}
and $\CXIII < \frac{1}{3}$. Let $0 < \eta < 1$. By Lemma \ref{GDUnif}, there exists $\varepsilon_2 > 0$ such that if $\diam(D) < \varepsilon_2$, then
$$
\sup_{y \in \RR} \int_D|\partial_zG(z,y)b(z)|dz \leq \eta.
$$
We put $\varepsilon = \min(\varepsilon_1, \varepsilon_2)$ and $\diam(D)\leq\varepsilon$. 
By Lemma \ref{lem:pf},
\begin{align}
\tG(x,y) &\leq  G(x,y) + \int_D \tG(x,z)|b(z)\partial_zG(z,y)|dz \label{ali1} \\ 
&\leq  G(x,y) + \CXVII\eta. \label{ali2} 
\end{align}
By putting the estimates of $\tG$ from \eqref{ali2} into \eqref{ali1} and applying \eqref{eq1:Theorem1s}, we get
$$
\tG(x,y) \leq  G(x,y) + \int_D (G(x,y) + \CXVII\eta)  |b(z)\partial_zG(z,y)|dz \le G(x,y)(1 + \CXIII) + \CXVII\eta^2.
$$
By induction,
\begin{equation}\label{Gseries}
\tG(x,y) \leq G(x,y)(1 + \CXIII + \dots + \CXIII^{n-1}) + \CXVII\eta^n.
\end{equation}
Now, taking $n\rightarrow\infty$, for every $x,y \in D$, we obtain 
\begin{align}\label{eq2:Theorem1s}
	\tG(x,y) \leq G(x,y)\frac{1}{1-\CXIII}.
\end{align}
Since $\CXIII < \frac{1}{3}$, by Lemma \ref{lem:pf}, \eqref{eq2:Theorem1s} and \eqref{eq1:Theorem1s}, we get
$$
\tG(x,y)\geq G(x,y) - \frac{1}{1-\CXIII} \int_D G(x,z)|b(z)\partial_zG(z,y)|dz \ge G(x,y)\left(1 - \frac{\CXIII}{1-\CXIII}\right).
$$
\end{proof}

We note that the comparison constants in the proof above will improve to $1$ if  ${\rm diam}(D)\to 0$ and the distortion of $D$ is bounded.
By (\ref{eq:IWt}),
\begin{equation}
  \tPP^x(X_{\tau_D} \in A) \stackrel{\CXVI}{\approx} \PP^x(X_{\tau_D} \in A)\,,\quad x\in D\,,\quad A\subset (\overline{D})^c\, , \label{I-WComp1}
\end{equation}
where $\CXVI = \CXVI(\underline{\sigma}, b, \lambda, \diam(D))$ and $\diam(D) < \varepsilon$ from Lemma \ref{Theorem1s}.

Following \cite[Proof of Lemma 14]{MR2892584}, we obtain that the boundary of our general $C^{1,1}$ open set $D$ is not hit at the first exit, i.e.
\begin{align}\label{l:nu}
	\tPP^x(X_{\tau_D}\in \partial D)=0, \qquad x \in D.
\end{align}
Hence, in the context of Lemma~\ref{Theorem1s}, 
the $\tPP^x$ distribution of $X_{\tau_D}$ is absolutely continuous with respect to the Lebesgue measure,
and has density function 
\begin{equation}\label{PoissonComp}
  \tilde P_D(x,y) \approx P_D(x,y)\,,\quad \;y\in D^c\,,
\end{equation}
provided $x\in D$.
This follows from (\ref{eq:IWt}) and \eqref{l:nu}. 

The definition of $\tilde{\mathcal{L}}$-harmonicity is analogous to that of $\mathcal{L}$-harmonicity
\begin{definition}
We say that a function $f:\R\rightarrow\R$ is $\tilde{\mathcal L}-$harmonic on an open bounded set $D\subset \R$, if for any open $F \subset \overline{F} \subset D$ and $x\in F$ 
$$
f(x) = \tilde\E^xf(X_{\tau_F}).
$$
We say that a function $f$ is regular $\tilde{\mathcal L}-$harmonic on an open bounded set $D\subset \R$,  if for every $x \in D$
$$
f(x) = \tilde\E^xf(X_{\tau_D}).
$$
\end{definition}

Following \cite{MR2892584} and \cite{2017-TG-TJ-GZ-pms}, we get the following Harnack inequality.

\begin{lem}[Harnack inequality for $\tL$]\label{HIforL} 
  Let $x,y \in \R$, $0<s<1$ and $k\in \mathbb{N}$ satisfy $|x-y|
  \le 2^ks$. Let ${u}$ be nonnegative in $\R$ and
  $\tilde{\mathcal{L}}$-harmonic in $B(x,s) \cup B(y,s)$.  There is $\CXVIII = \CXVIII(\ua, \uC, b)$ such that
  \begin{equation}\label{LHarnackIneq}
    \CXVIII^{-1}2^{-k(1+\ua)}{u}(x) \le {u}(y) \le \CXVIII 2^{k(1+\ua)}{u}(x)\,.
  \end{equation}
\end{lem}

We obtain a boundary Harnack principle for $\mathcal{L}$ and general $C^{1,1}$ sets $D$. See proof of \cite[Lemma 4.3]{2017-TG-TJ-GZ-pms}

\begin{lem}[BHP]\label{BHPforL} 
  Let $z \in \partial{D}$, $0<r\le  r_0(D)$, and $0<q<1$. If
$\tilde{u}, \tilde{v}$ are nonnegative in $\R$, regular $\tilde{\mathcal{L}}$-harmonic
  in $D \cap B(z,r)$, vanish on $D^c \cap B(z,r)$
  and satisfy $\tilde{u}(x_0)=\tilde{v}(x_0)$ for some $x_0 \in D \cap B(z,qr)$
  then
  \begin{equation}
    \label{BHPforLEq}
    \CXIX^{-1}\tilde{v}(x) \le \tilde{u}(x) \le \CXIX \tilde{v}(x)\,,\quad x \in D \cap B(z,qr)\,,
  \end{equation}
with $\CXIX = \CXIX(\sigma,b,q,r_0(D))$.
\end{lem}

Now, we have all the tools necessary to prove the main result of our paper. Since in the proof we follow the idea from \cite{MR2892584}, we only give its basic steps (for details see  \cite[Proof of Theorem 1]{MR2892584}).
\begin{proof}[Proof of Theorem {\ref{Theorem1}}]\label{sec:b}
 By (\ref{eq:wp}), we have the estimate
\begin{equation}\label{eq:0}
  \tG(x,y) \le G(x,y) + \int_{D} |\tG(x,z)\partial_zG(z,y)| |b(z)|\,dz\,,\quad
x,y\in D\,.
\end{equation}
We consider  $\eta<1$, say $\eta=1/2$.
By Lemma~\ref{GDUnif} 
there is a constant $r>0$ so small that
\begin{equation}\label{eq:1}
\int_{D^r} |\partial_zG(z,y)b(z)|\,dz  <\eta\,,\quad
  y \in D\,,
\end{equation}
and
\begin{equation}\label{eq:2}
  \int_{D^r} \frac{G(x,z)|\partial_zG(z,y)|}{G(x,y)} |b(z)|\,dz <\eta\,, \qquad y \in D\,,
\end{equation}
Where $D^r = \lbrace z\in D:\dez\leq r\rbrace$. We denote $$\rho=[\varepsilon \land r_0(D)\land r]/16\,,$$ with $\varepsilon=\varepsilon(\psi,b,\lambda, \diam(D))$ of Lemma~\ref{Theorem1s}. 

To prove (\ref{eq:egf}) we will consider $x$ and $y$ in a partitions of $D\times D$.

\noindent First, we consider $y$ far from the boundary of $D$, say $\dey \ge \rho/4$.
\begin{itemize}
	\item For $|x-y| \le \rho/8$, $G(x,y) \approx G_B(x,y) \approx \tG_B(x,y) \approx \tG_D(x,y)$ (we use Lemmas \ref{GreenEst1}, \ref{Theorem1s}, \ref{lem:GLDupper}).
	\item If $\rho/8<\dex$ we use Harnack inequalities for $\mathcal{L}$ and $\tilde{\mathcal{L}}$.
	\item For $\dex < \rho/8$ we use Boundary Harnack principles (see Lemma~\ref{BHPforL}, \cite[Theorem 2.18]{KSV2014}).
\end{itemize}

\noindent Next, suppose that $\deltaDD(y) \le \rho /4$. Here, the difficulty lies in the fact $\tG$ is non-symmetric.\\
In the proof of lower bounds we consider two cases: $x$ close to $y$ and $x$ far away from $y$.
\begin{itemize}
\item In the case $|x-y| \le \rho$, we locally approximate $D$ by the small $C^{1,1}$ set $F$ such that $\delta_D(x) = \delta_F(x)$ and  $\delta_D(x) = \delta_F(x)$ (see \cite[Lemma 1]{MR2892584}). Then $\tG(x,y) \ge \tG_F(x,y) \approx G_F(x,y) \approx G_D(x,y)$ (see Lemma \ref{GreenEst1}).\item For $|x-y| > \rho$ and $\delta_D(x) \ge \rho/4$ we use Harnack inequalities. For $\delta_D(x) \le \rho/4$ we use boundary Harnack principles.
\end{itemize}

\noindent In the next step, we prove the upper bound in (\ref{eq:egf}) for $\deltaDD(x) \ge \rho/4$.
We have already proved that for $y \in D \setminus D^r$,
$$
c_1^{-1} G(x,y) \le \tG(x,y) \le c_1 G(x,y)\,.
$$
By (\ref{def:kappa}), Lemma~\ref{lem:boundKappa}, Lemma~\ref{lem:GLDupper}, (\ref{eq:0}) and (\ref{eq:1}),
\begin{align}
  \tG(x,y)  & \le A G(x,y) + \int_{D_r} \tG(x,z)|\partial_zG(z,y)b(z)|\,dz\,, \label{eq:toi} \\
  & \le AG(x,y) + B(x)\,, \label{eq:3}
\end{align}
where $A = 1+c_1\CIII$ and $B(x)=\eta \CVII $.
Now, plugging (\ref{eq:3}) into
(\ref{eq:toi}), and using (\ref{eq:1}), (\ref{eq:2}) and induction, we get for $n=0,1,\ldots$,
\begin{equation}\label{eq:AB}
\tG_D(x,y) \le A\big(1 + \eta +\cdots + \eta^n \big)G_D(x,y) +
\eta^n B(x)\,.
\end{equation}
In consequence,
\begin{equation}\label{eq:ubi}
\tG_D(x,y) \le \frac{A}{1- \eta} G_D(x,y)\,.
\end{equation}
Finally, we prove the upper bound in (\ref{eq:egf}) when $\dex< \rho/4$.
\begin{itemize}
\item If $|x-y| > \rho$, we use boundary Harnack principles.
\item For $|x-y| \le \rho$, consider the same set $F$ as above. We have
\end{itemize}

$$
  \tGDD(x,y) =
   \tG_F(x,y) + \int_{D \setminus F} \tP_F(x,z) \tGDD(z,y) \,dz\,.
$$
By Lemma~\ref{Theorem1s} and (\ref{PoissonComp}), $\tG_F(x,y) \approx G_F(x,y)$ and $\tP_F(x,z)\approx P_F(x,z)$. We already know that  for $|z-y| > \rho$, $\tGDD(z,y) \approx G(z,y)$. Thus,
$$
\tGDD(x,y) \approx G_F(x,y) + \int_{D \setminus F} P_F(x,z)
 G_D(z,y) \,dz = G_D(x,y)\,.
$$
The proof of Theorem~\ref{Theorem1} is complete.
\end{proof}

\bibliographystyle{abbrv}

\end{document}